\documentclass{amsart}[12pt]

\usepackage{amsmath}
\usepackage{amsfonts}
\usepackage{amssymb}

\theoremstyle{definition}

 \newtheorem{thm}{Theorem}[section]
 \newtheorem{cor}[thm]{Corollary}
 \newtheorem{lem}[thm]{Lemma}
 \newtheorem{prop}[thm]{Proposition}
 \theoremstyle{definition}
 
 \theoremstyle{remark}

 \numberwithin{equation}{section}

\newcommand{\B}{{\mathcal{B}}}

\newcommand{\DD}{{\mathbb D}}

\def \uc {uC_\varphi}
\def \psi {u}

\def\bege{\begin{equation}} \def\ende{\end{equation}}
   
\def\b{\beta}   
 
\def\begr{\begin{eqnarray}} \def\endr{\end{eqnarray}}

\def\bege{\begin{equation}} \def\ende{\end{equation}}
\def\begr{\begin{eqnarray}} \def\endr{\end{eqnarray}}
\def\bnum{\begin{enumerate}} \def\enum{\end{enumerate}}

\begin{document}

\title[Norm and essential norm of a weighted composition operator ]
 {Norm and essential norm of  a weighted composition operator on the Bloch space}

\author[Xiaosong Liu]{Xiaosong Liu}

\address{%
    Xiaosong Liu:  Department of Mathematics, Jiaying University, Meizhou 514015, China;}

\email{gdxsliu@163.com}

\thanks{This work is  supported by NSF of China (No.11471143).}

\author{Songxiao Li$^\dagger$}\thanks{$^\dagger$Corresponding author.}
\address{Songxiao Li: Institute os System Engineering,  Macau University of Science and Technology, Avenida Wai Long, Taipa,
Macau.  }
\email{jyulsx@163.com}



\date{\today}

\begin{abstract}
We give some new estimates   for the norm and essential norm of  a weighted composition operator on the
Bloch space. As corollaries,  we obtain some new characterizations of the boundedness and compactness of a weighted composition operator on the
Bloch space.
\end{abstract}

\maketitle

\section{Introduction}

Let $\DD = \{z : |z| < 1\}$ be the unit disk in the complex plane
and $H(\DD)$ be the space of all analytic functions on $\DD$.  For
$a\in \DD$, let $\sigma_a$ be the automorphism of $\DD$ exchanging
$0$ for $a$, namely $\sigma_a(z)=\frac{a-z}{1-\bar{a}z}, ~z\in \DD.$
For $0<p<\infty$, the Bergman space $A^p$ consists of all  $f\in H(\DD)$ such that
\[
\|f\|_{A^p}^p=\int_{\DD}|f(z)|^pdA(z)<\infty,
\]
where $dA(z)=\frac{1}{\pi}dxdy$ denote the normalized area Lebesgue measure.   The Bloch space, denoted by $\mathcal{B}=\mathcal{B}(\DD)$,  is the space of all $f \in
H(\DD)$ such that
  \begr \|f\|_\b  = \sup_{z \in \DD}(1-|z|^2) |f'(z)|<\infty. \nonumber\endr
    Under the norm
$\|f\|_{\mathcal{B}}=|f(0)|+ \|f\|_{\b}$, the Bloch space is a Banach
space. From Theorem 1 of \cite{ax}, we see that
\[  \|f\|_\beta  \approx \sup_{a \in \DD}\|f\circ\sigma_a-f(a)\|_{A^2} .\]
See \cite{zhu2} for more information of the Bloch space.

For $0<p<\infty$, let $H^p$ denote the Hardy space of functions $f\in H(\DD)$ such that
 $$
\|f\|_{H^p}^p=\sup_{0<r<1}\frac{1}{2\pi}\int_0^{2\pi}|f(re^{i\theta})|^pd\theta<\infty.
$$
We say that an $f \in H(\DD)$ belongs to the $ BMOA$ space,  if
$$
\|f\|^2_{*} =\sup_{I\subseteq \partial \DD} \frac{1}{ |I|} \int_{I}|f(\zeta)-f_I|^2\frac{d\zeta}{2\pi} <\infty,
$$
where $f_I=\frac{1}{|I|}\int_If(\zeta)\frac{d\zeta}{2\pi}.$ It is well known that $BMOA$ is a Banach space under the norm $\|f\|_{BMOA}=|f(0)|+\|f\|_{*}$. From \cite{ga}, we have
  $$\|f\|_{*} \approx  \sup_{w\in \DD}\|f\circ\sigma_w-f(w)\|_{H^2}.
 $$

Throughout the paper, $S(\DD)$ denotes the set of all analytic self-maps of $\DD$.   Let $u \in H(\mathbb{D})$ and $\varphi\in S(\DD)$. For $f \in H(\mathbb{D})$, the   composition operator $C_\varphi$ and the  multiplication operator $M_u$ are defined by
$$
(C_\varphi f)(z) =    f(\varphi(z)) ~~~\mbox{and}~~~~~(M_u f)(z)=u(z)f(z),$$
respectively. The weighted composition operator $uC_\varphi$, induced by $u  $ and $\varphi $,    is defined as follows.
$$
(uC_\varphi f)(z) =u(z)  f(\varphi(z)), \  \  f \in H(\mathbb{D}). $$
It is clear that the weighted composition operator $uC_\varphi$ is the generalization of
$C_\varphi$ and $M_u$.

It is well known   that $C_\varphi$ is bounded on $BMOA$    for any $\varphi\in S(\DD)$ by Littlewood's subordination theorem. The compactness of the  operator $C_\varphi  :BMOA\rightarrow BMOA$ was studied in \cite{bcm, gll2013, smith, wulan, wzz}.  Based on results in \cite{bcm} and \cite{smith}, Wulan in \cite{wulan} showed that $C_\varphi  :BMOA\rightarrow BMOA$ is compact   if and only if
 \begr \lim_{n\rightarrow \infty}\| \varphi^n\|_*=0~~~~\,\,\,\,and\,\,\,\,  \lim_{|\varphi(a)|\rightarrow 1} \| \sigma_a  \circ\varphi\|_{*}=0.\nonumber
 \endr
   In \cite{wzz},  Wulan, Zheng and Zhu  further showed that $C_\varphi  :BMOA\rightarrow BMOA$ is compact   if and only if
    $ \lim_{n\rightarrow \infty} \|\varphi^n\|_*=0$.   In  \cite{Lj2007}, Laitila gave some function theoretic characterizations for the boundedness and compactness of the operator $uC_\varphi :BMOA\rightarrow BMOA$. In \cite{col}, Colonna used the idea of \cite{wzz} and showed that  $\uc :BMOA\rightarrow BMOA$  is compact  if and only if
\[
\lim_{n\rightarrow \infty}\|\psi\varphi^n\|_*=0\,\,\,\,and\,\,\,\,
 \lim_{|\varphi(a)|\rightarrow 1}\big(\log\frac{2}{1-|\varphi(a)|^2}\big)\|\psi\circ\sigma_a-\psi(a)\|_{H^2}=0.
\]
 Motivated by results in \cite{col},  Laitila and  Lindstr\"{o}m gave the   estimates for norm and essential norm of the
weighted composition $uC_\varphi :BMOA\rightarrow BMOA$ in \cite{ll}, among others, they showed that, under the assumption of the boundedness of $uC_\varphi$ on $BMOA$,
 \[\|uC_\varphi \|_{e,BMOA\rightarrow BMOA}\approx \limsup_{n\rightarrow \infty}\|\psi\varphi^n\|_*+\limsup_{|\varphi(a)|\rightarrow 1}\big(\log\frac{2}{1-|\varphi(a)|^2}\big)\|\psi\circ\sigma_a-\psi(a)\|_{H^2}.\]


 %

Recall that the essential norm of a bounded linear operator $T:X\rightarrow Y$ is its distance to the set of
compact operators $K$ mapping $X$ into $Y$, that is,
$$\|T\|_{e, X\rightarrow Y}=\inf\{\|T-K\|_{X\rightarrow Y}: K~\mbox{is compact}~~\},$$ where  $X,Y$ are Banach spaces and
$\|\cdot\|_{X\rightarrow Y}$ is the operator norm.

 By Schwarz-Pick lemma, it is easy to see   that $C_\varphi$ is bounded on the Bloch space $\mathcal{B}$  for any $\varphi\in S(\DD)$. The compactness of $C_\varphi$   on
$\mathcal{B}$ was studied in for example \cite{lou, mm, t, wzz, zhao}.  In \cite{wzz},  Wulan, Zheng and Zhu proved that
 $C_\varphi  :\B\rightarrow \B$ is compact if and only if $\lim_{n\rightarrow\infty}\| \varphi^n \|_\B=0.$  In \cite{zhao}, Zhao obtained the exact value for the essential norm of
$C_\varphi  :\B \rightarrow\B $ as follows.
$$
\|C_\varphi\|_{e,\B \rightarrow \B }  =\Big(\frac{e}{2 }\Big)  \limsup_{n\rightarrow\infty}   \|\varphi^n \|_{\B}.
 $$
  In \cite{osz},  Ohno, Stroethoff and Zhao studied the boundedness and compactness of the operator $u C_\varphi:\mathcal{B} \to \mathcal{B}  $. In \cite{col1}, Colonna provided a new characterization of the boundedness and compactness of the operator $u C_\varphi:\mathcal{B} \to \mathcal{B}  $ by using $\|u\varphi^n\|_\B$.  The essential norm of the operator $u C_\varphi:\mathcal{B}  \to \mathcal{B}  $ was studied in \cite{h1, mz1, mz2}. In \cite{mz1},   the authors proved that
$$
\|u C_\varphi\|_{e,\B\rightarrow \B}  $$
$$\approx\max\Big( \limsup_{|\varphi(a)|\rightarrow 1}  \frac{|u(z)\varphi'(z)|(1-|z|^2) }{ 1-|\varphi(z)|^2 }
 , ~~ \limsup_{|\varphi(a)|\rightarrow 1} \log\frac{e}{1-|\varphi(a)|^2} |u'(z)| (1-|z|^2)  \Big)
  $$
 In \cite{h1}, the authors obtained a new estimate for the essential norm of $u C_\varphi:\mathcal{B} \to \mathcal{B}  $, i.e., they showed that
$$
\|u C_\varphi\|_{e,\B\rightarrow \B } \approx \max\Big( \limsup_{j\rightarrow\infty}
 \|I_u (\varphi^j )\|_{\B }, ~~  \limsup_{j\rightarrow\infty}
\log n\|J_u (\varphi^j )\|_{\B } \Big),
$$
where $ I_u f(z)=\int_0^z f'(\zeta)u(\zeta)d\zeta,  J_u f(z)=\int_0^z f(\zeta)u'(\zeta)d\zeta.$

 Motivated by the work of \cite{col1, col, ll, wzz}, the aim of this article is to give some new estimates   for the norm and essential norm of the operator  $u C_\varphi:\mathcal{B} \to \mathcal{B}  $. As corollaries, we obtain some new characterizations for the boundedness and compactness of the operator  $u C_\varphi:\mathcal{B} \to \mathcal{B}  $.

Throughout this paper, constants are denoted by $C$, they are positive and may differ from one occurrence to the other. The
notation $a \lesssim b$ means that there is a positive constant $C$ such that $a \leq C b$. Moreover, if both $a\lesssim b$ and $b\lesssim
a$ hold, then one says that $a \approx b$.

\section{Norm of $uC_\varphi $ on the Bloch space}
In this section we give some estimates for the norm of the operator  $u C_\varphi:\mathcal{B} \to \mathcal{B}  $. For this purpose, we need some lemmas which stated as follows. The following lemma can be found in \cite{zhu2}.

\begin{lem} {\it Let $f\in\B$. Then
\[
|f(z)|\lesssim  \log\frac{2}{1-|z|^2}\|f\|_\B , ~~~~~z\in\DD.\]}
\end{lem}

\begin{lem}{\it
For $2\leq p< \infty$ and $f\in \B$,
 \[
\sup_{a \in \DD}\|f\circ\sigma_a-f(a)\|_{A^2}\approx \sup_{a \in \DD}\|f\circ\sigma_a-f(a)\|_{A^p}.
\]}
\end{lem}

\begin{proof} Using  H\"{o}lder inequality,  we get
\begr
\sup_{a \in \DD}\|f\circ\sigma_a-f(a)\|_{A^2}\leq \sup_{a \in \DD}\|f\circ\sigma_a-f(a)\|_{A^p},
\endr
for $2\leq p <\infty$.

On the other hand,  there exists a constant
$C >0$ such that (see \cite[p.38]{Xi})
 \begr \sup_{a \in \DD}\|f\circ\sigma_a-f(a)\|_{A^p} \leq
C \|f\|_{\B} \lesssim \sup_{a \in \DD}\|f\circ\sigma_a-f(a)\|_{A^2},
\endr
which, combined with (2.1),  implies the desired result.  \end{proof}

\begin{lem} \cite{Sw1996} {\it
For $f\in A^2$,
\[
\|f\|_{A^2}^2 \approx |f(0)|^2+\int_{\DD}|f'(w)|^2(1-|w|^2)^2dA(w).
\]}
\end{lem}

The classical Nevanlinna counting function  $N_\varphi$ and the generalized
Nevanlinna counting functions  $N_{\varphi,\gamma}$ for $\varphi$ are defined by (see \cite{Sj1987})
\[
N_\varphi(w)=\sum_{z\in \varphi^{-1}\{w\}}\log\frac{1}{|z|}\,\,\,and\,\,\,
N_{\varphi,\gamma}(w)=\sum_{z\in \varphi^{-1}\{w\}}\big(\log\frac{1}{|z|}\big)^\gamma,
\]
respectively, where $\gamma>0$ and $w\in\DD\backslash \{\varphi(0)\}$.

\begin{lem} \cite{Sw1996} {\it Let $\varphi\in S(\DD)$ and $f\in A^2$. Then
\[
\|f\circ\varphi\|_{A^2}^2 \approx |f(\varphi(0))|^2+\int_{\DD}|f'(w)|^2N_{\varphi,2}(w)dA(w).
\]}
\end{lem}

\begin{lem}\cite{Sj1987} {\it Let $\varphi\in S(\DD)$ and $\gamma>0$. If $\varphi(0)\neq 0$ and $0<r<|\varphi(0)|$, then
\[
N_{\varphi,\gamma}(0)\leq \frac{1}{r^2}\int_{r\DD}N_{\varphi,\gamma}dA.
\]}
\end{lem}

 \begin{lem} {\it
Let $\varphi\in S(\DD)$ such that $\varphi(0)=0$. If $\sup_{0<|w|<1}|w|^2N_{ \varphi,2}(w)<\delta,$ then
 \begr
N_{ \varphi,2}(w) \leq  \frac{4\delta}{(\log 2)^2}\big(\log \frac{1}{|w|} \big)^2
\endr
when $\frac{1}{2}\leq |w|<1$.}
\end{lem}

\begin{proof} See the proof of Lemma 2.1 in \cite{smith}.\end{proof}

\begin{lem} {\it For all $g\in A^2$ and $\phi\in S(\DD)$ such $g(0)=\phi(0)=0$, we have
\begr
\|g\circ \phi\|_{A^2} \lesssim  \|\phi\|_{A^2}\|g\|_{A^2}.
\endr
 In particular, for all $f\in\B$, $a\in\DD$ and $\varphi\in S(\DD)$,
\begr
      \|f\circ\varphi\circ\sigma_a-f(\varphi(a))\|_{A^2} \lesssim    \|\sigma_{\varphi(a)}\circ\varphi\circ\sigma_a\|_{A^2} \|f\circ\sigma_a-f(a)\|_{A^2}. \nonumber
\endr  }
 \end{lem}

\begin{proof}
Let $\phi\in S(\DD)$ such that  $\phi(0)=0$. Then,
\begr
\|\sigma_z\circ \phi-\sigma_z(\phi(0))\|_{A^2}^2=
\int_{\DD}\frac{(1-|z|^2)^2|\phi(w)|^2}{|1-\bar{z}\phi(w)|^2}dA(w)
\leq 4\|\phi\|_{A^2}^2.
\endr
From Lemmas 2.3 and 2.4,  and (2.5) we obtain
\begr
\|\sigma_z\circ \phi-\sigma_z(\phi(0))\|_{A^2}^2&=&
\int_{\DD}|(\sigma_z\circ\phi )'|^2(\log\frac{1}{|w|})^2dA(w)\nonumber\\
&=&\int_{\DD} N_{\sigma_z\circ\phi,2} dA(w)\leq 4\|\phi\|_{A^2}^2.
\endr
For $z\in\DD\setminus\{0\}$, from Lemma 4.2 in \cite{Sw1996} and Lemma 2.5, we have
\begr
|z|^2N_{\phi,2}(z)= |z|^2N_{\sigma_z\circ\phi,2}(0) \leq \int_{|z|\DD} N_{\sigma_z\circ\phi,2}(w) dA(w)\leq 4\|\phi\|_{A^2}^2.
\endr
So, by Lemma 2.6 we get
\begr
 N_{\phi,2}(z)  \leq \frac{16}{(\log2)^2}\|\phi\|_{A^2}^2(\log\frac{1}{|z|})^2,
\endr
for $z\in\DD\setminus \frac{1}{2}\DD $. Thus,
\begr
\int_{\DD\setminus \frac{1}{2}\DD }|g'(z)|^2N_{\phi,2}(z) dA(z)\leq \frac{16}{(\log2)^2}\|\phi\|_{A^2}^2\|g\|_{A^2}^2.
\endr
In addition, for $z\in\DD$ and $g\in A^2$, from Theorems 4.14 and 4.28 of \cite{zhu2}, we have $|g'(z)|\leq (1-|z|^2)^{-2}\|g\|_{A^2}$. Then,
\begr
&&\int_{\frac{1}{2}\DD}|g'(z)|^2N_{\phi,2}(z)dA(z)\leq 16\|g\|_{A^2}^2\int_{\frac{1}{2}\DD}N_{\phi,2}(z)dA(z)
\leq16\|\phi\|_{A^2}^2\|g\|_{A^2}^2.
\endr
Since  $g(0)=0$,  by Lemma 2.4 we have
\begr
\|g\circ \phi\|_{A^2}^2 \approx \int_{\DD}|g'(z)|^2N_{\phi,2}(z) dA(z).
\endr
Combine with (2.9), (2.10) and (2.11), we   obtain
\begr
\|g\circ \phi\|_{A^2} \lesssim  \|\phi\|_{A^2}\|g\|_{A^2} ,\nonumber
\endr
as desired.  In particular, for all $f\in\B$, $a\in\DD$ and $\varphi\in S(\DD)$, if we set $$g= f\circ \sigma_{\varphi(a)}-f(\varphi(a)), ~~~\phi=\sigma_{\varphi(a)} \circ \varphi \circ \sigma_a,$$
 we get
 \begr
      \|f\circ\varphi\circ\sigma_a-f(\varphi(a))\|_{A^2} \lesssim    \|\sigma_{\varphi(a)}\circ\varphi\circ\sigma_a\|_{A^2} \|f\circ\sigma_a-f(a)\|_{A^2}. \nonumber
\endr
   The proof is complete.  \end{proof}

For the simplicity of the rest of this paper, we introduce the following abbreviation. Set
 \[
\alpha(\psi,\varphi,a)=|\psi(a)|\cdot\|\sigma_{\varphi(a)}\circ\varphi\circ\sigma_a\|_{A^2},
\]
\[
\beta(\psi,\varphi,a)= \log\frac{2}{1-|\varphi(a)|^2} \|\psi\circ\sigma_a-\psi(a)\|_{A^2},
\]
where $a\in \DD$, $\psi\in H(\DD)$ and $\varphi\in S(\DD)$.

\begin{thm} {\it
Let $\psi\in H(\DD)$ and $\varphi\in S(\DD)$. Then
\begr
\|uC_\varphi\|_{\B \rightarrow \B}\approx |\psi(0)|\log\frac{2}{1-|\varphi(0)|^2}+
\sup_{a\in \DD}\alpha(\psi,\varphi,a)+\sup_{a\in \DD}\beta(\psi,\varphi,a).\nonumber
\endr}
\end{thm}

\begin{proof}
First we give the upper estimate for $\|uC_\varphi\|_{\B \rightarrow \B}$.  For all $f\in \B$, using the triangle inequality, we get
\begr
&&\|(uC_\varphi f)\circ\sigma_a-(uC_\varphi f)(a)\|_{A^2}\nonumber\\
&=&\|(\psi\circ\sigma_a-\psi(a))\cdot(f\circ\varphi\circ\sigma_a-f(\varphi(a)))\nonumber\\
&&+\psi(a)(f\circ\varphi\circ\sigma_a-f(\varphi(a)))
+(\psi\circ\sigma_a-\psi(a))f(\varphi(a))\|_{A^2}\nonumber\\
&\leq&\|(\psi\circ\sigma_a-\psi(a))\cdot(f\circ\varphi\circ\sigma_a-f(\varphi(a)))\|_{A^2}\nonumber\\
&&+|\psi(a)|\|f\circ\varphi\circ\sigma_a-f(\varphi(a))\|_{A^2}+
|f(\varphi(a))|\|\psi\circ\sigma_a-\psi(a)\|_{A^2}.
\endr
 By Lemmas 2.1 and 2.7, we have
\begr
&&|\psi(a)|\|f\circ\varphi\circ\sigma_a-f(\varphi(a))\|_{A^2}+
|f(\varphi(a))|\|\psi\circ\sigma_a-\psi(a)\|_{A^2}\nonumber\\
&\lesssim&  \alpha(\psi,\varphi,a)|\|f\circ\sigma_a-f(a)\|_{A^2}+
\log\frac{2}{1-|\varphi(a)|^2}\|\psi\circ\sigma_a-\psi(a)\|_{A^2}  \|f \|_{\B}\nonumber \\
 &\lesssim& \big(\alpha(\psi,\varphi,a)+\beta(\psi,\varphi,a) \big) \|f \|_{\B} .
\endr
From Lemmas 2.1 and 2.2,  we get
\begr
&&\sup_{a\in \DD}\|(\psi\circ\sigma_a-\psi(a))\cdot(f\circ\varphi\circ\sigma_a-f(\varphi(a)))\|_{A^2}\nonumber\\
&\lesssim & \sup_{a\in \DD}\log2\|\psi\circ\sigma_a-\psi(a)\|_{A^2}\|f\circ\varphi\circ\sigma_a-f(\varphi(a))\|_{A^2}\nonumber\\
&\lesssim & \sup_{a\in \DD}\log\frac{2}{1-|\varphi(a)|^2}\|\psi\circ\sigma_a-\psi(a)\|_{A^2}\|f\circ\varphi\circ\sigma_a-f(\varphi(a))\|_{A^2}\nonumber\\
&\lesssim& \sup_{a\in \DD}\beta(\psi,\varphi,a)\|f\circ \varphi\|_{\B} \lesssim \sup_{a\in \DD}\beta(\psi,\varphi,a)\|f\|_{\B}.
\endr
Then, by (2.12), (2.13) and (2.14), we have
\begr
\sup_{a\in \DD}\|(uC_\varphi f)\circ\sigma_a-(uC_\varphi f)(a)\|_{A^2}\lesssim \big(\sup_{a\in \DD}\alpha(\psi,\varphi,a)
+\sup_{a\in \DD}\beta(\psi,\varphi,a)\big)\|f\|_{\B}.\nonumber
\endr
In addition, by Lemma 2.1, $|(uC_\varphi f)(0)|\lesssim |\psi(0)|\log\frac{2}{1-|\varphi(0)|^2}\|f\|_{\B}, $  we get
    \begr
 \|uC_\varphi f\|_\B &\approx & |(uC_\varphi f)(0)|+ \sup_{a\in \DD}\|(uC_\varphi f)\circ\sigma_a-(uC_\varphi f)(a)\|_{A^2} \nonumber\\
&\lesssim &  |\psi(0)|\log\frac{2}{1-|\varphi(0)|^2}\|f\|_{\B}+
\sup_{a\in \DD}\alpha(\psi,\varphi,a)\|f\|_{\B}+\sup_{a\in \DD}\beta(\psi,\varphi,a)\|f\|_{\B}\nonumber,
\endr
 which implies
   \begr
&&\|uC_\varphi\|_{\B \rightarrow \B}\lesssim |\psi(0)|\log\frac{2}{1-|\varphi(0)|^2}+
\sup_{a\in \DD}\alpha(\psi,\varphi,a)+\sup_{a\in \DD}\beta(\psi,\varphi,a).
\endr

Next we find the lower estimate for $\|uC_\varphi\|_{\B \rightarrow \B}$.  Let $f=1$. It is easy to see that $\|\psi\|_{\B}\leq \|uC_\varphi\|_{\B \rightarrow \B}$. For any $a\in \DD$, set
\begr
 f_a(z)=\sigma_{\varphi(a)}(z)-\varphi(a),\,\,\,\,\,\,\, z\in\DD.
\endr
Then, $f_a(0)=0$, $f_a(\varphi(a))=-\varphi(a)$, $\|f_a\|_{\B}\leq 4$ and $\|f_a\|_{\infty}\leq 2$. Using triangle inequality, we get
\begr
\alpha(\psi,\varphi,a)&=&|\psi(a)|\cdot\|\sigma_{\varphi(a)}\circ\varphi\circ\sigma_a-\varphi(a)+\varphi(a) \|_{A^2}\nonumber\\
&=& \|\psi(a)\cdot(f_a\circ\varphi\circ\sigma_a-f_a(\varphi(a)))\|_{A^2}\nonumber   \\
&\leq&\|(\psi\circ\sigma_a-\psi(a))\cdot f_a\circ\varphi\circ\sigma_a\|_{A^2}\nonumber\\
&&+\|(\psi\circ\sigma_a)\cdot f_a\circ\varphi\circ\sigma_a-\psi(a)f_a(\varphi(a))\|_{A^2}\\
&\leq&2\|\psi\circ\sigma_a-\psi(a)\|_{A^2}+\|(uC_\varphi f_a)\circ\sigma_a-(uC_\varphi f_a)(a)\|_{A^2}\nonumber\\
&\leq&2\|u \|_\B+4\|uC_\varphi\|_{\B \rightarrow \B}  \leq 6\|uC_\varphi\|_{\B \rightarrow \B} \nonumber.
\endr
Set
\begr
 h_a(z)=\log\frac{2}{1-\overline{\varphi(a)}z},\,\,\,\,\,\, z\in\DD.
\endr
Then, $h_a\in \B$, $h_a(\varphi(a))=\log\frac{2}{1-|\varphi(a)|^2}$ and   $\sup_{a\in \DD} \|h_a \|_{\B} \leq 2+\log 2$.
Using triangle inequality and Lemma 2.7, we obtain
\begr
\beta(\psi,\varphi,a)&=& \|\log\frac{2}{1-|\varphi(a)|^2}\cdot (\psi\circ\sigma_a-\psi(a))\|_{A^2}\nonumber\\
&=& \|h_a(\varphi(a))(\psi\circ\sigma_a-\psi(a))\|_{A^2}\nonumber\\
&\leq&\|(h_a\circ\varphi\circ\sigma_a-h_a(\varphi(a)))\cdot (\psi\circ\sigma_a-\psi(a))\|_{A^2}\nonumber\\
&&+ \| (\psi\circ\sigma_a)\cdot h_a\circ\varphi\circ\sigma_a- \psi(a)h_a(\varphi(a))\|_{A^2} \\
&&+ \|\psi(a) ( h_a\circ\varphi\circ\sigma_a- h_a(\varphi(a))    )  \|_{A^2} \nonumber\\
& \lesssim &\|(h_a\circ\varphi\circ\sigma_a-h_a(\varphi(a)))\cdot (\psi\circ\sigma_a-\psi(a))\|_{A^2}\nonumber\\
&&+\|(uC_\varphi h_a)\circ\sigma_a-(uC_\varphi h_a)(a)\|_{A^2}+ \alpha(\psi,\varphi,a)\|h_a\circ\sigma_a-h_a(a)\|_{A^2}\nonumber\\
& \lesssim &\|(h_a\circ\varphi\circ\sigma_a-h_a(\varphi(a)))\cdot (\psi\circ\sigma_a-\psi(a))\|_{A^2}\nonumber \\
&&+(2+\log 2)\|uC_\varphi  \|_{\B \rightarrow \B} + (2+\log 2)\alpha(\psi,\varphi,a)\nonumber .
\endr
By Lemmas 2.2 and 2.7, we have
\begr
&&\|(h_a\circ\varphi\circ\sigma_a-h_a(\varphi(a)))\cdot (\psi\circ\sigma_a-\psi(a))\|_{A^2}\nonumber\\
& \lesssim &  \|(h_a\circ\varphi\circ\sigma_a-h_a(\varphi(a)))\|_{A^2}\|\psi\circ\sigma_a-\psi(a)\|_{A^2} \nonumber\\
&\leq & \|h_a\circ\varphi \|_{\B} \|\psi\|_{\B}
\lesssim \|uC_\varphi\|_{\B \rightarrow \B}.
\endr
Combining (2.17), (2.19) and (2.20), we have
\begr
\sup_{a\in \DD}\alpha(\psi,\varphi,a)+\sup_{a\in \DD}\beta(\psi,\varphi,a)\lesssim \|uC_\varphi\|_{\B \rightarrow \B}\nonumber.
\endr
Moreover, since
\begr
|\psi(0)|\log\frac{2}{1-|\varphi(0)|^2}=|(uC_\varphi h_{0})(0)| \leq  (2+\log 2)  \|uC_\varphi\|_{\B \rightarrow \B} \lesssim  \|uC_\varphi\|_{\B \rightarrow \B}\nonumber.
\endr
Therefore,
\begr
 &&|\psi(0)|\log\frac{2}{1-|\varphi(0)|^2}+\sup_{a\in \DD}\alpha(\psi,\varphi,a)+\sup_{a\in \DD}\beta(\psi,\varphi,a)
\lesssim  \|uC_\varphi\|_{\B \rightarrow \B}.\nonumber
\endr
 We complete the proof of the theorem.
\end{proof}

As a corollary, we obtain the following new characterization of the boundedness of $uC_\varphi:\B \rightarrow \B.$

\begin{cor} {\it
Let $\psi\in H(\DD)$ and $\varphi\in S(\DD)$. Then  $uC_\varphi:\B \rightarrow \B$ is bounded  if and only if
\begr
 \sup_{a\in \DD}|\psi(a)|\cdot\|\sigma_{\varphi(a)}\circ\varphi\circ\sigma_a \|_{A^2}  <\infty \,\,\,\nonumber\endr
 and
   \begr \sup_{a\in \DD} \log\frac{2}{1-|\varphi(a)|^2} \|\psi\circ\sigma_a-\psi(a)\|_{A^2}   <\infty\nonumber.
\endr}
\end{cor}

In particular, when $\varphi(z)=z$, we obtain the estimate of the norm of the multiplication $M_u: \B \rightarrow \B$.
\begin{cor} {\it
Let $\psi\in H(\DD)$. Then
 $$
  \|M_u\|_{\B \rightarrow \B} \approx |\psi(0)|\log2+\sup_{a\in \DD}  \log\frac{2}{1-|a|^2} \|\psi\circ\sigma_a-\psi(a)\|_{A^2} .
  $$}
\end{cor}

\begin{lem} {\it Suppose that $uC_\varphi:\B \rightarrow \B$ is bounded.  Then
\begr
 \sup_{a\in \DD}\|uC_\varphi (\sigma_{\varphi(a)}-\varphi(a)) \|_{\B} \approx   \sup_{n\geq 0}\| \psi\varphi^n \|_{\B}
\endr
and
\begr
\limsup_{|\varphi(a)|\rightarrow 1}\|  uC_\varphi (\sigma_{\varphi(a)}-\varphi(a)) \|_{\B} \lesssim\limsup_{n\rightarrow \infty}\| \psi\varphi^n \|_{\B} .
\endr}
\end{lem}

\begin{proof} From Corollary 2.1 of \cite{col1}, we see that
\begr
 \sup_{a\in \DD}\|uC_\varphi \sigma_{\varphi(a)} \|_{\B} \approx   \sup_{n\geq 0}\| \psi\varphi^n \|_{\B}. \nonumber
\endr
Then (2.21)   follows immediately.

The Taylor expansion of $\sigma_{\varphi(a)}-\varphi(a)$ is
\begr
 \sigma_{\varphi(a)}-\varphi(a)=-\sum_{n=0}^\infty \big(\overline{\varphi(a)}\big)^n(1-|\varphi(a)|^2)z^{n+1}\nonumber.
\endr
Then, by the boundedness of  $uC_\varphi:\B \rightarrow \B$   we have
\begr
&& \| uC_\varphi (\sigma_{\varphi(a)}-\varphi(a)) \|_\B   \leq   (1-|\varphi(a)|^2)\sum_{n=0}^\infty |\varphi(a)|^n\|\psi\varphi^{n+1} \|_{\B}\nonumber
 \endr
For each $N$, set
\[M_1=:\sum_{n=0}^N|\varphi(a)|^n\|\psi\varphi^{n+1} \|_{\B}.\]
  Then we get
\begr
&&\| uC_\varphi (\sigma_{\varphi(a)}-\varphi(a)) \|_\B\nonumber\\
&\leq& (1-|\varphi(a)|^2)\sum_{n=0}^N |\varphi(a)|^n\|\psi\varphi^{n+1} \|_{\B} +\big((1-|\varphi(a)|^2)\sum_{n=N+1}^\infty |\varphi(a)|^n\big)\|\psi\varphi^{n+1} \|_{\B}\nonumber\\
&\leq& M_1(1-|\varphi(a)|^2) +\big((1-|\varphi(a)|^2)\sum_{n=N+1}^\infty |\varphi(a)|^n \sup_{n\geq N+1}\|\psi\varphi^{n+1} \|_{\B}\nonumber\\
&\leq& M_1(1-|\varphi(a)|^2)+2\sup_{n\geq N+1}\|\psi\varphi^{n+1} \|_{\B}\nonumber.
\endr
Taking $\limsup_{|\varphi(a)|\rightarrow 1}$ to the last inequality and then letting $N\rightarrow\infty$, we get the desired result.
\end{proof}

\begin{prop} {\it
Let $\varphi\in S(\DD)$ and $\psi\in H(\DD)$. The following statements hold.
\begin{itemize}
\item[(i)] For $a\in \DD$, let $f_a(z)=\sigma_{\varphi(a)}-\varphi(a)$. Then
\begr
\alpha(\psi,\varphi,a)\lesssim \frac{\beta(\psi,\varphi,a)}{\log\frac{2}{1-|\varphi(a)|^2}}+\|(uC_\varphi f_a)\circ\sigma_a-(uC_\varphi f_a)(a)\|_{A^2}\nonumber.
\endr
\item[(ii)] For $a\in \DD$, let $g_a=\frac{h^2_a}{h_a(\varphi(a))}$, where $h_a(z)=\log\frac{2}{1-\overline{\varphi(a)}z}$. Then
\begr
\beta(\psi,\varphi,a)&\lesssim &\alpha(\psi,\varphi,a)+\|(g_a\circ\varphi\circ\sigma_a-g_a(\varphi(a)))\cdot(\psi\circ\sigma_a-\psi(a))\|_{A^2}\nonumber\\
&&+  \|(uC_\varphi g_a)\circ\sigma_a-(uC_\varphi g_a)(a)\|_{A^2}\nonumber.
\endr
\item[(iii)] For all $f\in\B$ and $a\in\DD$,
\begr
\|(uC_\varphi f)\circ\sigma_a-(uC_\varphi f)(a)\|_{A^2}&\lesssim& \|(\psi\circ\sigma_a-\psi(a))\cdot(f\circ\varphi\circ\sigma_a-f(\varphi(a)))\|_{A^2}\nonumber\\
&&+\Big(\alpha(\psi,\varphi,a)+\beta(\psi,\varphi,a)\Big)\|f\|_{\B}\nonumber.
\endr
\item[(iv)]For all $f\in\B$ and $a\in\DD$,
\begr
&&\|(\psi\circ\sigma_a-\psi(a))\cdot(f\circ\varphi\circ\sigma_a-f(\varphi(a)))\|_{A^2}\nonumber\\
&\lesssim&\|f\|_{\B} \min \Big\{\sup_{w\in\DD}\beta(\psi,\varphi,w),\frac{\|uC_\varphi\|_{\B\rightarrow \B}}{\sqrt{\log\frac{2}{1-|\varphi(a)|^2} }}\Big\}\nonumber.
\endr
\end{itemize}}
\end{prop}

\begin{proof}
(i) It is easy to see that $\|f_a\circ\varphi\circ\sigma_a\|_\infty\leq 2$. For any $a\in\DD$, we get
\begr
\alpha(\psi,\varphi,a)&=&|\psi(a)|\|f\circ\varphi\circ\sigma_a-f(\varphi(a))\|_{A^2}\nonumber\\
&=&\|(\psi\circ\sigma_a-\psi(a))\cdot f_a\circ\varphi\circ\sigma_a-(uC_\varphi f_a)\circ\sigma_a-(uC_\varphi f_a)(a)\|_{A^2}\nonumber\\
&\lesssim& \|\psi\circ\sigma_a-\psi(a)\|_{A^2}+\|(uC_\varphi f_a)\circ\sigma_a-(uC_\varphi f_a)(a)\|_{A^2}\nonumber\\
&\leq&  \frac{\beta(\psi,\varphi,a)}{\log\frac{2}{1-|\varphi(a)|^2}}+\|(uC_\varphi f_a)\circ\sigma_a-(uC_\varphi f_a)(a)\|_{A^2}\nonumber.
\endr

(ii) It is obvious that $g_a(\varphi(a))=\log\frac{2}{1-|\varphi(a)|^2}$. Since  $(g_a\circ\sigma_{\varphi(a)}-g_a(\varphi(a)))(0)=0$,
\[
g_a\circ\varphi\circ\sigma_a-g_a(\varphi(a))=g_a\circ\sigma_{\varphi(a)}\circ(\sigma_{\varphi(a)}
\circ\varphi\circ\sigma_a)-g_a(\varphi(a)),
\]
   by Lemma 2.7 and the fact that $\sup_{a\in\DD}\|g_a \|_{\B}<\infty$  we obtain
\[
|\psi(a)|\|g_a\circ\varphi\circ\sigma_a-g_a(\varphi(a))\|_{A^2}\lesssim \alpha(\psi,\varphi,a)\sup_{a\in\DD}\|g_a \|_{\B} \lesssim \alpha(\psi,\varphi,a).
\]
 By the triangle inequality we get
\begr
\beta(\psi,\varphi,a)&=& \|g_a(\varphi(a))\cdot (\psi\circ\sigma_a-\psi(a)) \|_{A^2}\nonumber\\
&=&\|(g_a\circ\varphi\circ\sigma_a-g_a(\varphi(a)))\cdot(\psi\circ\sigma_a-\psi(a))\nonumber\\
&&+\psi(a)(g_a\circ\varphi\circ\sigma_a-g_a(\varphi(a)))-(\psi(a)g_a\circ\varphi\circ\sigma_a-\psi(a)g_a(\varphi(a)))\|_{A^2}\nonumber\\
&\leq&\|(g_a\circ\varphi\circ\sigma_a-g_a(\varphi(a)))\cdot(\psi\circ\sigma_a-\psi(a))\|_{A^2}\nonumber\\
&&+|\psi(a)|\|g_a\circ\varphi\circ\sigma_a-g_a(\varphi(a))\|_{A^2}+ \|(uC_\varphi g_a)\circ\sigma_a-(uC_\varphi g_a)(a)\|_{A^2}\nonumber\\
&\lesssim &\alpha(\psi,\varphi,a)+\|(g_a\circ\varphi\circ\sigma_a-g_a(\varphi(a)))\cdot(\psi\circ\sigma_a-\psi(a))\|_{A^2}\nonumber\\
&&+  \|(uC_\varphi g_a)\circ\sigma_a-(uC_\varphi g_a)(a)\|_{A^2}\nonumber,
\endr
as desired.

(iii) See the proof of Theorem 2.8.

(iv) Using the fact that  $\log 2\leq \log\frac{2}{1-|\varphi(a)|^2}$ and Theorem 2.8, we have
 \begr
\sup_{a\in\DD}\|\psi\circ\sigma_a-\psi(a)\|_{A^2}\leq \sup_{a\in\DD}\beta(\psi,\varphi,a)\lesssim \|uC_\varphi\|_{\B\rightarrow \B} .
\endr
 By Lemma 2.2 and H\"{o}lder inequality, we obtain
\begr
&&\|(\psi\circ\sigma_a-\psi(a))\cdot(f\circ\varphi\circ\sigma_a-f(\varphi(a)))\|_{A^2}^2\nonumber\\
&=&\|(\psi\circ\sigma_a-\psi(a))^2(f\circ\varphi\circ\sigma_a-f(\varphi(a)))^2\|_{A^1}\nonumber\\
&\leq&\|\psi\circ\sigma_a-\psi(a)\|_{A^2}\|\psi\circ\sigma_a-\psi(a)\|_{A^4}\|f\circ\varphi\circ\sigma_a-f(\varphi(a))\|_{A^8}^2\nonumber\\
&\leq&\|\psi\circ\sigma_a-\psi(a)\|_{A^2}\sup_{a\in\DD}\|\psi\circ\sigma_a-\psi(a)\|_{A^4}
\sup_{a\in\DD}\|f\circ\varphi\circ\sigma_a-f(\varphi(a))\|_{A^8}^2\nonumber\\
&\lesssim& \beta(\psi,\varphi,a)\sup_{a\in\DD}\|\psi\circ\sigma_a-\psi(a)\|_{A^2} \sup_{a\in\DD}\|f\circ\varphi\circ\sigma_a-f(\varphi(a))\|_{A^2}^2/\log\frac{2}{1-|\varphi(a)|^2}\nonumber.
\endr
Then, by the boundedness of $C_\varphi$ on $\B$ and (2.23), we obtain
\begr
&&\beta(\psi,\varphi,a)\sup_{a\in\DD}\|\psi\circ\sigma_a-\psi(a)\|_{A^2}
\sup_{a\in\DD}\|f\circ\varphi\circ\sigma_a-f(\varphi(a))\|_{A^2}^2/\log\frac{2}{1-|\varphi(a)|^2}\nonumber\\
&\lesssim&(\sup_{a\in\DD}\beta(\psi,\varphi,a))^2\sup_{a\in\DD}\|f\circ\varphi\circ\sigma_a-f(\varphi(a))\|_{A^2}^2/\log\frac{2}{1-|\varphi(a)|^2}\nonumber\\
&\lesssim&\sup_{a\in\DD}\|f\circ\varphi\circ\sigma_a-f(\varphi(a))\|_{A^2}^2\min \bigg\{\sup_{a\in\DD}\beta(\psi,\varphi,a),
\frac{\|uC_\varphi\|_{\B\rightarrow \B}}{\sqrt{\log\frac{2}{1-|\varphi(a)|^2}}}\bigg\}^2\nonumber\\
&\lesssim& \|f\circ\varphi \|_{\B}^2\min \bigg\{\sup_{a\in\DD}\beta(\psi,\varphi,a),
\frac{\|uC_\varphi\|_{\B\rightarrow \B}}{\sqrt{\log\frac{2}{1-|\varphi(a)|^2}}}\bigg\}^2\nonumber\\
&\lesssim& \|f  \|_{\B}^2 \min \bigg\{\sup_{a\in\DD}\beta(\psi,\varphi,a),
\frac{\|uC_\varphi\|_{\B\rightarrow \B}}{\sqrt{\log\frac{2}{1-|\varphi(a)|^2}}}\bigg\}^2\nonumber.
\endr
The proof is complete.
\end{proof}

\begin{thm} {\it
Let $\psi\in H(\DD)$ and $\varphi\in S(\DD)$. Suppose that $uC_\varphi$ is bounded on $\B$. Then
\begr
\|uC_\varphi \|_{\B \rightarrow \B}&\approx & |\psi(0)|\log\frac{2}{1-|\varphi(0)|^2} +
\sup_{n\geq 0}\|\psi\varphi^n \|_{\B}+\sup_{a\in \DD} \beta(\psi,\varphi,a)\nonumber.
\endr}
\end{thm}

\begin{proof} For any $f\in \B$, by (iii) and (iv) of Proposition 2.12, we get
\begr
\|uC_\varphi f\|_{\beta}\lesssim \sup_{a\in\DD}\big(\alpha(\psi,\varphi,a)
+\beta(\psi,\varphi,a)\big)\|f\|_{\B}.\nonumber
\endr
By Lemma 2.11 and (i) of Proposition 2.12, we have
\begr
\alpha(\psi,\varphi,a)&\lesssim& \beta(\psi,\varphi,a)/\log\frac{2}{1-|\varphi(a)|^2}  + \sup_{a\in\DD}\|uC_\varphi f_a\|_{\B}\nonumber \\
& \lesssim & \beta(\psi,\varphi,a)+ \sup_{n\geq 0} \|\psi\varphi^n\|_{\B}.\nonumber
\endr
Thus,
\begr
\|uC_\varphi f\|_{\beta}   \lesssim  \big(\sup_{a\in\DD} \beta(\psi,\varphi,a)
+\sup_{n\geq 0}\|\psi\varphi^n\|_{\B}\big)\|f\|_{\B}\nonumber.
\endr
In addition, $(uC_\varphi f)(0)|=|\psi(0)||f(\varphi(0))|\lesssim |\psi(0)||\log\frac{2}{1-|\varphi(0)|^2}\|f\|_{\B}\nonumber.$ Therefore,
 \begr
\|uC_\varphi\|_{\B \rightarrow \B}\lesssim |\psi(0)|\log\frac{2}{1-|\varphi(0)|^2}+\sup_{n\geq 0}\|\psi\varphi^n\|_{\B}+\sup_{a\in \DD}\beta(\psi,\varphi,a).
\nonumber
\endr

On the other hand, let $f(z)=z^n$. Then $f\in\B$ for all $n\geq 0$.
Thus
\[\sup_{n\geq 0}\|\psi\varphi^n\|_{\B}=\sup_{n\geq 0}\|(uC_\varphi)z^n\|_{\B}\leq\|uC_\varphi\|_{\B \rightarrow \B} <\infty, \]
which, together with Theorem 2.8, implies
\begr
|\psi(0)|\log\frac{2}{1-|\varphi(0)|^2}+\sup_n\|\psi\varphi^n\|_{\B}+\sup_{a\in \DD}\beta(\psi,\varphi,a)
\lesssim\|uC_\varphi\|_{\B \rightarrow \B}\nonumber.
\endr
The proof of the theorem is complete.
\end{proof}

\begin{cor} {\it
Let $\psi\in H(\DD)$ and $\varphi\in S(\DD)$. Then  $uC_\varphi:\B \rightarrow \B$ is bounded if and only if
\begr
 \sup_{n\geq 0}\|\psi\varphi^n \|_{\B}<\infty   ~~~~~\mbox{and}~~~~~~~~~~\sup_{a\in \DD}  \log\frac{2}{1-|\varphi(a)|^2} \|\psi\circ\sigma_a-\psi(a)\|_{A^2}<\infty
\nonumber.
\endr}
\end{cor}

\section{Essential norm of $uC_\varphi $ on the Bloch space}
In this section we characterize the essential norm of the weighted composition operator $uC_\varphi:\B \rightarrow \B$
  by using various form, especially we will use the Bloch norm of $\psi \varphi^n$. For $t\in(0,1)$, we define
\begr
E(\varphi,a,t)=\{z\in\DD:|(\sigma_{\varphi(a)}\circ\varphi\circ\sigma_a)(z)|>t\} \nonumber.
\endr
Similarly to the proof of Lemma 9 of \cite{ll}, we get the following result. Since the proof is similar, we omit the details.

\begin{lem} {\it
Let $\psi\in H(\DD)$ and $\varphi\in S(\DD)$. Then
\begr
&&\widetilde{\gamma}:=\limsup_{r\rightarrow 1}\limsup_{t\rightarrow 1}\sup_{|\varphi(a)|\leq r}
\Big(\int_{E(\varphi,a,t)}|\psi(\sigma_a(z))|^4dA(z)\Big)^{1/4} \lesssim
\limsup_{n\rightarrow \infty}\|\psi\varphi^n\|_{\B}\nonumber.
\endr}
\end{lem}

\begin{thm} {\it
Let $\psi\in H(\DD)$ and $\varphi\in S(\DD)$ such that $uC_\varphi:\B \rightarrow \B$ is bounded. Then
\begr
\|uC_\varphi\|_{e,\B\rightarrow \B}&\approx& \limsup_{n\rightarrow\infty}\|\psi\varphi^n\|_{\B}
+\limsup_{|\varphi(a)|\rightarrow 1}\|\uc g_a\|_\B\nonumber\\
&\approx&  \widetilde{\alpha}+   \widetilde{\beta}+ \widetilde{\gamma}\nonumber\\
&\approx&  \widetilde{\alpha}+ \limsup_{|\varphi(a)|\rightarrow 1}\|\uc g_a\|_\B  + \widetilde{\gamma} \nonumber\\
&\approx&  \limsup_{n\rightarrow\infty}\|\psi\varphi^n\|_{\B}  + \widetilde{\beta} \nonumber,
\endr
where  $\widetilde{\alpha }=\limsup_{|\varphi(a)|\rightarrow 1} \alpha(\psi,\varphi,a), ~~~~~ \widetilde{\beta}=\limsup_{|\varphi(a)|\rightarrow 1} \beta(\psi,\varphi,a)$ and
$$ ~~~g_a(z)= \Big(\log\frac{2}{1- \overline{\varphi(a)}z} \Big)^2    \Big(\log\frac{2}{1-|\varphi(a)|^2}\Big)^{-1}.$$}
\end{thm}

\begin{proof}
 Set $f_n(z)=z^n$. It is well known that $f_n\in\B$  and $f_n\rightarrow 0$ weakly in $\B$ as $n\rightarrow \infty$. Then
\begr
\|uC_\varphi\|_{e,\B\rightarrow \B} \gtrsim \limsup_{n\rightarrow\infty} \|uC_\varphi f_n\|_{\B}=\limsup_{n\rightarrow\infty}\|\psi\varphi^n\|_{\B}.
\endr
Choose $a_n\in\DD$ such that $|\varphi(a_n)|\rightarrow 1$ as $n\rightarrow \infty$. It is easy to check $g_{a_n}$ are uniformly bounded in
$\B$  and converges weakly to zero in $\B$ (see \cite{osz}). By these facts we obtain
\begr
\|uC_\varphi\|_{e,\B\rightarrow \B} &\gtrsim& \limsup_{n\rightarrow \infty} \|uC_\varphi g_{a_n}\|_{\B} = \limsup_{|\varphi(a)|\rightarrow 1} \|uC_\varphi g_a\|_{\B}.
\endr
By (3.1) and (3.2), we obtain
\begr
\|uC_\varphi\|_{e,\B\rightarrow \B}\gtrsim \limsup_{n\rightarrow\infty}\|\psi\varphi^n\|_{\B}
+\limsup_{|\varphi(a)|\rightarrow 1}\|\uc g_a\|_\B.
\endr
From (i) of Proposition 2.12, we see that
\begr
\alpha(\psi,\varphi,a) \lesssim \frac{\beta(\psi,\varphi,a)}{\log\frac{2}{1-|\varphi(a)|^2}}+\|uC_\varphi f_a \|_{\B}, \nonumber
\endr
which together with Lemma 2.11 implies that
 \begr
\widetilde{\alpha }=\limsup_{|\varphi(a)|\rightarrow 1} \alpha(\psi,\varphi,a) \lesssim \limsup_{|\varphi(a)|\rightarrow 1}\|uC_\varphi f_a \|_{\B}  \lesssim \limsup_{n\rightarrow\infty}\|\psi\varphi^n\|_{\B}.
\endr
From (ii) and (iv) of  Proposition 2.12, we see that
 \begr
\beta(\psi,\varphi,a)&\lesssim &\alpha(\psi,\varphi,a)+\|(g_a\circ\varphi\circ\sigma_a-g_a(\varphi(a)))\cdot(\psi\circ\sigma_a-\psi(a))\|_{A^2}\nonumber\\
&&+  \|(uC_\varphi g_a)\circ\sigma_a-(uC_\varphi g_a)(a)\|_{A^2} \nonumber\\
& \lesssim & \alpha(\psi,\varphi,a)+ \|g_a\|_{\B}\frac{\|uC_\varphi\|_{\B\rightarrow \B}}{\sqrt{\log\frac{2}{1-|\varphi(a)|^2} }} + \|uC_\varphi g_a \|_{\B},\nonumber
\endr
 which implies that
  \begr
  \widetilde{\beta}=\limsup_{|\varphi(a)|\rightarrow 1} \beta(\psi,\varphi,a)& \lesssim &  \widetilde{\alpha} + \limsup_{|\varphi(a)|\rightarrow 1} \|uC_\varphi g_a \|_{\B} .
  \endr
By Lemma 3.1, (3.3), (3.4) and (3.5), we have
 \begr
  \|uC_\varphi\|_{e,\B\rightarrow \B}&\gtrsim&   \widetilde{\alpha} +\widetilde{\gamma}+\limsup_{|\varphi(a)|\rightarrow 1}\|\uc g_a\|_\B\nonumber\\
 &\gtrsim&   \widetilde{\alpha} +\widetilde{\gamma}+ \widetilde{\beta}, \nonumber
 \endr
and
 \begr
  \|uC_\varphi\|_{e,\B\rightarrow \B} &\gtrsim &  \limsup_{n\rightarrow\infty}\|\psi\varphi^n\|_{\B} + \widetilde{\alpha} +\limsup_{|\varphi(a)|\rightarrow 1}\|\uc g_a\|_\B     \nonumber\\
  &\gtrsim   &\limsup_{n\rightarrow\infty}\|\psi\varphi^n\|_{\B}+  \widetilde{\beta}  \nonumber  .
 \endr

Next we give the upper estimate for $\|uC_\varphi\|_{e, \B \rightarrow \B}$.
For $n\geq 0$, we define the linear operator on $\B$ by $(K_nf)(z)=f( \frac{n}{n+1}z)$.   It is easy to check that $K_n$ is a compact operator on $\B$. Thus
\[
\|uC_\varphi\|_{e,\B\rightarrow \B}\leq \limsup_{n \rightarrow \infty}\sup_{\|f\|_{\B}\leq 1}\|uC_\varphi(I-K_n)f\|_{\B},
\]
where $I$ is the identity operator. Let $S_n=I-K_n$. Then,
\begr
\|uC_\varphi \|_{e,\B\rightarrow \B}
&\leq& \liminf_{n\rightarrow \infty}\|uC_\varphi S_n\|_{\B}\nonumber\\
&=&\liminf_{n\rightarrow \infty}\sup_{\|f\|_{\B}\leq 1}\big(|\psi(0)(S_nf)(\varphi(0))| + \|(uC_\varphi S_nf \|_{\beta}\big)\\
&=&\liminf_{n\rightarrow \infty}\sup_{\|f\|_{\B}\leq 1} \| uC_\varphi S_nf \|_{\beta}\nonumber.
\endr
 Let $f\in\B$ such that $\|f\|_{\B}\leq 1$. Fix $n\geq 0$, $r\in(0,1)$ and $t\in(\frac{1}{2},1)$. Then
\begr
\| uC_\varphi S_nf \|_{\beta}&\approx&\sup_{a\in \DD}\|(uC_\varphi S_nf)\circ\sigma_a-(uC_\varphi S_nf)(a)\|_{A^2}\nonumber\\
&\leq& \sup_{|\varphi(a)|\leq r}\|(uC_\varphi S_nf)\circ\sigma_a-(uC_\varphi S_nf)(a)\|_{A^2}\nonumber\\
&&+\sup_{|\varphi(a)|> r}\|(uC_\varphi S_nf)\circ\sigma_a-(uC_\varphi S_nf)(a)\|_{A^2}.
\endr
By (iii) and (iv) of Proposition 2.11, we have
\begr
&&\sup_{|\varphi(a)|> r}\|(uC_\varphi S_nf)\circ\sigma_a-(uC_\varphi S_nf)(a)\|_{A^2}\nonumber\\
&\lesssim& \| S_nf \|_{\B}\sup_{|\varphi(a)|> r}\bigg(\alpha(\psi,\varphi,a)+\beta(\psi,\varphi,a)
+\frac{\|uC_\varphi \|_{\B\rightarrow\B}}{ \sqrt{\log\frac{2}{1-|\varphi(a)|^2}}}\bigg) .
\endr
In addition,
\begr
&&\sup_{|\varphi(a)|\leq r}\|(uC_\varphi S_nf)\circ\sigma_a-(uC_\varphi S_nf)(a)\|_{A^2}\nonumber\\
&\leq&\sup_{|\varphi(a)|\leq r}\big(|(S_nf)(\varphi(a))|\|\psi\circ\sigma_a-\psi(a)\|_{A^2} \nonumber\\
 && +\|\psi\circ\sigma_a\cdot (( C_\varphi S_nf)\circ\sigma_a-( C_\varphi S_nf)(a))\|_{A^2}\big)\nonumber\\
&\leq&\|\psi \|_{\B}\max_{|w|\leq r}|(S_nf)(w)|+I_1^{1/2}+I_2^{1/2} ,
\endr
where
\begr
I_1=\sup_{|\varphi(a)|\leq r}\int_{\DD\backslash E(\varphi,a,t)}|(\psi\circ\sigma_a)(z)\cdot((S_nf)\circ\varphi\circ\sigma_a(z)-(S_nf)(\varphi(a)))|^2dA(z), \nonumber
\endr
\begr
I_2=\sup_{|\varphi(a)|\leq r}\int_{ E(\varphi,a,t)}|(\psi\circ\sigma_a)(z)\cdot((S_nf)\circ\varphi\circ\sigma_a(z)
-(S_nf)(\varphi(a)))|^2dA(z).\nonumber
\endr
Let $\varphi_a=\sigma_{\varphi(a)}\circ\varphi\circ \sigma_a$.
Then by (3.19) in \cite[p.37]{Lj2007}, we have
\begr |(S_nf)\circ\sigma_{\varphi(a)}\circ\varphi_a(z)-(S_nf\circ\varphi)(a)|  \lesssim   \sup_{|w|\leq t}|
\big((S_nf)\circ\sigma_{\varphi(a)}\big)(w)-(S_nf)(\varphi(a))|\nonumber
\endr
for $z\in \DD\backslash E(\varphi,a,t)$. Since
 \begr
\|\psi\circ\sigma_a\cdot\varphi_a\|_{A^2}&\leq&\|\psi\circ\sigma_a-\psi(a)\|_{A^2}\|\varphi_a\|_\infty
+|\psi(a)|\|\varphi_a\|_2\nonumber\\
&\lesssim&\sup_{a\in\DD}\|\psi\circ\sigma_a-\psi(a)\|_{A^2}+\alpha(\psi,\varphi,a_n)\nonumber\\
&\lesssim& \|uC_\varphi\|_{\B\rightarrow \B}\nonumber,
\endr
we have
\begr
I_1& \lesssim& \sup_{|\varphi(a)|\leq r}\sup_{|w|\leq t}|\big((S_nf)\circ\sigma_{\varphi(a)}\big)(w)-(S_nf)(\varphi(a))|^2
\|\psi\circ\sigma_a\cdot\varphi_a\|_{A^2}^2\nonumber\\
&\lesssim & \|uC_\varphi\|^2_{\B\rightarrow \B} \sup_{|z|\leq\frac{t+r}{1+tr}}|(S_nf)(z)|^2\nonumber.
\endr
By Lemma 2.2,  we get
\begr
\|(S_nf)\circ\varphi\circ\sigma_a-(S_nf)(\varphi(a))\|_{A^4}^2&\leq&
\sup_{a\in\DD}\|(S_nf)\circ\varphi\circ\sigma_a-(S_nf)(\varphi(a))\|_{A^4}^2\nonumber\\
&\lesssim& \sup_{a\in\DD}\|f\circ\sigma_a-f(a)\|_{A^2}^2\leq 1,\nonumber
\endr
which implies that
\begr
I_2&\leq&\sup_{|\varphi(a)|\leq r}\Big(\int_{E(\varphi,a,t)}|\psi(\sigma_a(z))|^4dA(z)\Big)^{1/2}
 \|(S_nf)\circ\varphi\circ\sigma_a-(S_nf)(\varphi(a))\|_{A^4}^2\nonumber\\
&\leq&\sup_{|\varphi(a)|\leq r}\Big(\int_{E(\varphi,a,t)}|\psi(\sigma_a(z))|^4dA(z)\Big)^{1/2}. \nonumber
\endr
  By combining the above estimates, for $r\in(0,1)$ and $t\in(\frac{1}{2},1)$, we obtain
\begr
&& \| uC_\varphi S_nf \|_{\beta}  \nonumber\\
&\lesssim& \sup_{|\varphi(a)|> r}\Big(\alpha(\psi,\varphi,a)+\beta(\psi,\varphi,a)
+\frac{\|uC_\varphi \|_{\B\rightarrow\B}}{\sqrt{ \log\frac{2}{1-|\varphi(a)|^2} }}\Big)\nonumber\\
&&+\sup_{|\varphi(a)|\leq r}\Big(\int_{E(\varphi,a,t)}|\psi(\sigma_a(z))|^4dA(z)\Big)^{1/4}+\sup_{|z|\leq\frac{t+r}{1+tr}}|(S_nf)(z)| \|\uc\|_{\B\rightarrow\B}\nonumber.
\endr
Taking the supremum over $\|f\|_\B\leq 1$ and letting $n\rightarrow \infty$, we obtain
\begr
\|uC_\varphi\|_{e,\B\rightarrow \B} &\lesssim& \sup_{|\varphi(a)|> r}\Big(\alpha(\psi,\varphi,a)+\beta(\psi,\varphi,a)
+\frac{\|uC_\varphi \|_{\B\rightarrow\B}}{\sqrt{ \log\frac{2}{1-|\varphi(a)|^2} }}\Big)\nonumber\\
&&+\sup_{|\varphi(a)|\leq r}\Big(\int_{E(\varphi,a,t)}|\psi(\sigma_a(z))|^4dA(z)\Big)^{1/4},  \nonumber
\endr
which implies that
 \begr
\|uC_\varphi\|_{e,\B\rightarrow \B} \lesssim \widetilde{\alpha}+   \widetilde{\beta}+ \widetilde{\gamma}, \nonumber
\endr
 By (3.4) (3.5) and Lemma 3.1, we get
 \begr
\|uC_\varphi\|_{e,\B\rightarrow \B} &\lesssim & \widetilde{\beta}+  \limsup_{n\rightarrow \infty}\| \psi\varphi^n \|_{\B}\nonumber\\
 &\lesssim & \widetilde{\alpha}+   \limsup_{|\varphi(a)|\rightarrow 1}\|\uc g_a\|_\B+ \limsup_{n\rightarrow \infty}\| \psi\varphi^n \|_{\B}\nonumber\\
 &\lesssim &    \limsup_{|\varphi(a)|\rightarrow 1}\|\uc g_a\|_\B+ \limsup_{n\rightarrow \infty}\| \psi\varphi^n \|_{\B}. \nonumber
\endr
 By (ii), (iv) of Proposition 2.12, we have
\begr
&&\|uC_\varphi\|_{e,\B\rightarrow \B}\nonumber\\
 &\lesssim& \sup_{|\varphi(a)|> r}\Big(\alpha(\psi,\varphi,a)+\beta(\psi,\varphi,a)
+\frac{\|uC_\varphi \|_{\B\rightarrow\B}}{\sqrt{ \log\frac{2}{1-|\varphi(a)|^2} }}\Big)\nonumber\\
&&+\sup_{|\varphi(a)|\leq r}\Big(\int_{E(\varphi,a,t)}|\psi(\sigma_a(z))|^4dA(z)\Big)^{1/4} \nonumber\\
&\lesssim& \sup_{|\varphi(a)|> r}\Big(\alpha(\psi,\varphi,a)
+\|(uC_\varphi g_a)\circ\sigma_a-(uC_\varphi g_a)(a)\|_{A^2}\nonumber \\
&&+ \|(\psi\circ\sigma_a-\psi(a) ) \cdot  (g_a\circ\varphi\circ\sigma_a-g_a(\varphi(a)))\|_{A^2} + \frac{\|uC_\varphi \|_{\B\rightarrow\B}}{\sqrt{\log\frac{2}{1-|\varphi(a)|^2}}}\Big) \nonumber\\
&&+\Big(\int_{E(\varphi,a,t)}|\psi(\sigma_a(z))|^4dA(z)\Big)^{1/4} \nonumber\\
&\lesssim& \sup_{|\varphi(a)|> r}\Big(\alpha(\psi,\varphi,a)
+\| uC_\varphi g_a \|_{\B} + \|g_a\|_{\B} \frac{\|uC_\varphi \|_{\B\rightarrow\B}}{\sqrt{\log\frac{2}{1-|\varphi(a)|^2}}}  + \frac{\|uC_\varphi \|_{\B\rightarrow\B}}{\sqrt{\log\frac{2}{1-|\varphi(a)|^2}}}\Big) \nonumber\\
&&+\Big(\int_{E(\varphi,a,t)}|\psi(\sigma_a(z))|^4dA(z)\Big)^{1/4} , \nonumber
\endr
which implies that
 \begr
\|uC_\varphi\|_{e,\B\rightarrow \B} &\lesssim & \widetilde{\alpha}+   \limsup_{|\varphi(a)|\rightarrow 1}\|\uc g_a\|_\B+  \widetilde{\gamma}. \nonumber
\endr
We complete the proof of the theorem.
\end{proof}

From Theorem 3.2,  we immediately get the following characterizations for the compactness of the operator $uC_\varphi:\B \rightarrow \B$.
\begin{cor}  {\it
Let $\psi\in H(\DD)$ and $\varphi\in S(\DD)$ such that  $\uc$ is bounded on $\B$. Then the following statements are equivalent.
\begin{itemize}
\item[(i)] The operator $\uc:\B \rightarrow \B$ is compact.
\item[(ii)]$$\limsup_{n\rightarrow \infty}\|u\varphi^n\|_\B=0~~~~\mbox{and}~~~~  \limsup_{|\varphi(a)|\rightarrow 1}\|\uc g_a\|_\B=0.$$
\item[(iii)]$$\limsup_{n\rightarrow \infty}\|u\varphi^n\|_\B=0~~~~\mbox{and}~~~~  \limsup_{|\varphi(a)|\rightarrow 1}\beta(\psi,\varphi,a)=0.$$
\item[(iv)]\begr
 \limsup_{|\varphi(a)|\rightarrow 1}\alpha(\psi,\varphi,a)=0, ~~~~~ \limsup_{|\varphi(a)|\rightarrow 1}\beta(\psi,\varphi,a)=0,\nonumber
\endr
and
\begr
\limsup_{r\rightarrow 1}\limsup_{t\rightarrow 1}\sup_{|\varphi(a)|\leq r}
\Big(\int_{E(\varphi,a,t)}|\psi(\sigma_a(z))|^4dA(z)\Big)^{1/4} =0.\nonumber
\endr
\item[(v)]\begr
 \limsup_{|\varphi(a)|\rightarrow 1}\alpha(\psi,\varphi,a)=0, ~~~~~ \limsup_{|\varphi(a)|\rightarrow 1}\|\uc g_a\|_\B =0,\nonumber
\endr
and
\begr
\limsup_{r\rightarrow 1}\limsup_{t\rightarrow 1}\sup_{|\varphi(a)|\leq r}
\Big(\int_{E(\varphi,a,t)}|\psi(\sigma_a(z))|^4dA(z)\Big)^{1/4} =0.\nonumber
\endr
\end{itemize}}
\end{cor}

\end{document}